\documentclass[a4paper,12pt]{article}

\usepackage{amsmath}
\usepackage{amsthm}
\usepackage{enumerate}
\usepackage{amssymb,amsfonts,latexsym,mathtools}
\usepackage{mathrsfs}
\usepackage{bm}
\usepackage{cases}
\usepackage{color}

\newcommand{\Levy}{L\'{e}vy}

\newcommand{\R}{\mathbb{R}}

\newcommand{\Z}{\mathbb{Z}}
\newcommand{\Q}{\mathbb{Q}}
\newcommand{\e}{\varepsilon}

\renewcommand{\P}{\mathbb{P}}
\usepackage{multirow}
\newcommand{\E}{\mathbb{E}}
\usepackage{url}

\numberwithin{equation}{section}

\makeatletter
\renewcommand\section{\@startsection {section}{1}{\z@}%
{-3.5ex \@plus -1ex \@minus -.2ex}%
{2.3ex \@plus.2ex}%
{\normalfont\large\bf}}
\makeatother

\makeatletter
\renewcommand\subsection{\@startsection {subsection}{1}{\z@}%
{-3.5ex \@plus -1ex \@minus -.2ex}%
{2.3ex \@plus.2ex}%
{\normalfont\normalsize\bf}}
\makeatother

\theoremstyle{plain}
\newtheorem{thm}{Theorem}[section]
\newtheorem{lem}[thm]{Lemma}
\newtheorem{cor}[thm]{Corollary}
\newtheorem{prop}[thm]{Proposition}
\theoremstyle{definition}
\newtheorem{Rem}[thm]{Remark}

\allowdisplaybreaks

\usepackage{hyperref}
\hypersetup{
setpagesize=false,
 bookmarksnumbered=true,
 bookmarksopen=true,
 colorlinks=true,
 linkcolor=blue,
 citecolor=red,
}

\begin{document}
\begin{center}
\Large \textbf{Finite-State Time-Inhomogeneous Markov Chains Conditioned to Survive}
\end{center}
\begin{center}
Kohki Iba\footnote{
\begin{tabular}[t]{@{}l@{}}
Affiliation: Graduate School of Science, The University of Osaka, Osaka, Japan.\\
E-mail: \url{kohki.iba@gmail.com}
\end{tabular}
}
\end{center}
\begin{abstract}
We study conditioning on long survival for time-inhomogeneous Markov chains on finite state spaces. We assume that a sequence of substochastic transition kernels converges to an irreducible substochastic kernel and that the transition probabilities admit asymptotic expansions up to order $n^{-2}$. Under these assumptions, we construct a positive space-time harmonic function and prove that the limiting path distributions under conditioning on long survival are given by the corresponding Doob h-transform. We also investigate the long-time behavior of the resulting conditioned process. Finally, we apply our general theory to an elephant random walk, a generalized P\'{o}lya urn model, time-inhomogeneous Wright--Fisher and Moran models, and noncolliding random walks.
\end{abstract}




\section{Introduction}
\subsection{Overview}
Understanding the behavior of a stochastic process conditioned on an event of probability zero in the underlying path space is one of the fundamental problems in probability theory. In particular, considerable attention has been devoted to Markov chains killed upon hitting an absorbing state or exiting a prescribed domain, with a focus on their behavior under conditioning on survival.

When a time-homogeneous Markov chain is defined on a finite state space, such conditioning problems are closely related to the Perron--Frobenius theory for substochastic transition kernels. Under suitable irreducibility assumptions, the exponential decay rate of the survival probability is characterized by the principal eigenvalue of the transition kernel, and the limiting process obtained by conditioning on long survival can be constructed through the Doob h-transform associated with a corresponding positive eigenfunction. Such conditioning problems arise in a wide variety of stochastic models, including simple random walks, population genetics models, and systems of noncolliding particles. For example, simple random walks conditioned to stay in an interval have been studied by Debs \cite{Debs-conditioning}, while conditioning on nonabsorption in time-homogeneous Wright--Fisher and Moran models has been investigated by Huillet \cite{Huillet,Huillet-Moran}. For noncolliding random walks confined to a finite interval, exact path counts and their asymptotic behavior have been studied by Grabiner \cite{Grabiner} and Krattenthaler \cite{Krattenthaler}. Champagnat and Villemonais \cite{Champagnat-Villemonais-2} developed a general probabilistic framework for studying time-homogeneous Markov processes conditioned on nonabsorption. In the continuous-time setting, Brownian motion conditioned to stay in an interval has been studied by Knight \cite{Knight}, while related conditioning problems for spectrally negative \Levy\ processes have been investigated by Lambert \cite{Lambert}.

In Markov chains arising from stochastic models of real-world phenomena or from stochastic processes with memory, transition probabilities often depend on time. In the time-inhomogeneous setting, the limiting behavior under conditioning on long survival generally cannot be characterized solely by an eigenfunction of a single transition kernel. Instead, one needs to consider harmonic functions depending on both time and space, which makes the analysis considerably more involved. Champagnat and Villemonais \cite{Champagnat-Villeminais} established convergence of survival probability ratios, representations of conditioned processes via positive space-time harmonic functions, and quantitative mixing estimates. However, explicit asymptotic formulas for the survival probabilities themselves and first-order asymptotic expansions of the conditioned transition kernels are not derived in that work.

In this paper, we study this conditioning problem in a general framework for time-inhomogeneous Markov chains on finite state spaces. More precisely, we consider a sequence of substochastic transition kernels $(K_n)_{n\geq 0}$ incorporating killing, which converges to an irreducible substochastic kernel $K$, and assume that asymptotic expansions up to order $\frac{1}{n^2}$ are available for all allowed transitions. Our main objectives are to derive precise asymptotic formulas for survival probabilities and to construct the limiting process obtained by conditioning on long survival. To this end, we construct a positive space-time harmonic function $h$ and show that, for each fixed finite time horizon, the law of the path conditioned on survival up to a distant time converges to that of the process defined by the corresponding Doob h-transform. This provides a construction of time-inhomogeneous Markov chains conditioned never to be killed. We also investigate the long-time behavior of the resulting conditioned processes.

Finally, we apply our general theory to several concrete stochastic models. We consider the following processes:
\begin{itemize}
  \item An elephant random walk conditioned to stay in an interval.
  \item A generalized P\'{o}lya urn model conditioned so that the difference between the largest and smallest color counts remains bounded by a fixed constant.
  \item Wright--Fisher and Moran models conditioned on nonabsorption.
  \item Time-inhomogeneous noncolliding random walks conditioned to stay in an interval.
\end{itemize}
Each of these examples is of considerable interest in its own right and deserves a more detailed investigation. However, since such an investigation would be beyond the scope of the present paper, we restrict ourselves to verifying that each model satisfies the assumptions of our general theory.

\subsection{Notation and Assumptions}
Let $E$ be a finite set, and let $K$ be a nonzero, irreducible substochastic kernel on $E$. Assume that
\begin{align}
\sum_{y\in E}K(x,y)<1\qquad \text{for\ some}\ x\in E.
\end{align}
By the Perron--Frobenius theorem, $K$ has a positive eigenvalue $\lambda$ and corresponding strictly positive left and right eigenvectors $L$ and $R$, respectively, normalized so that
\begin{align}
  LK=\lambda L,\qquad KR=\lambda R,\qquad L,R>0,\qquad  LR=1.
\end{align}
Since $K$ admits killing, we have $0<\lambda<1$. We define the h-transform of $K$ with respect to $R$ by
\begin{align}
Q(x,y):=\frac{R(y)K(x,y)}{\lambda R(x)}\qquad \text{for}\ x,y\in E.
\end{align}
We denote by $\mathbf{Q}_{(m,x)}$ the law of the time-homogeneous Markov chain with transition kernel $Q$, starting from state $x\in E$ at time $m\ge 0$. We write $\zeta$ for its lifetime. Since $Q$ is stochastic, we have
\begin{align}
\label{Q-lifetime}
\mathbf{Q}_{(m,x)}(n<\zeta)=1\qquad \text{for}\ n\ge m.
\end{align}
The kernel $Q$ has a unique invariant distribution $\pi$, given by
\begin{align}
  \pi(x):=L(x)R(x)\qquad \text{for}\ x\in E.
\end{align}
Let $d\ge 1$ be the period of $Q$, and let $E_0,\ldots,E_{d-1}$ be its cyclic classes. These classes are pairwise disjoint and satisfy
\begin{align}
  Q(x,E_{r+1})=1\qquad \text{for}\ x\in E_r,
\end{align}
where $E_d:=E_0$. We define
\begin{align}
  \pi_r(y):=\frac{\pi(y)}{\pi(E_r)}1_{E_r}(y)\qquad \text{for}\ y\in E,\ r=0,...,d-1.
\end{align}
The indices of the cyclic classes and the corresponding distributions $\pi_r$ are understood modulo $d$. The measure $\pi_r$ is the probability distribution obtained by conditioning $\pi$ on $E_r$.

Let $(K_n)_{n\geq 0}$ be a sequence of nonzero, irreducible substochastic kernels on $E$ satisfying the following assumptions:
\begin{itemize}
  \item For all $n\ge 0$ and $x,y\in E$,
  \begin{align}
  \label{H1}
    K_n(x,y)>0\qquad \Longleftrightarrow\qquad K(x,y)>0.
    \tag{H1}
  \end{align}
  \item There exist functions $a,b:E\times E\to \R$ and an integer $N\ge 1$ such that for all $x,y\in E$ with $K(x,y)>0$, 
  \begin{align}
  \label{H2}
    \frac{K_n(x,y)}{K(x,y)}=1+\frac{a(x,y)}{n}+\frac{b(x,y)}{n^2}+O(n^{-3})\qquad \text{for}\ n\ge N.
    \tag{H2}
  \end{align}
  Here, the $O(n^{-3})$ term is uniform in $x$ and $y$.
\end{itemize}

Throughout this paper, all $O(\cdot)$ terms are understood to be uniform in the state variables.

For $m\ge 0$ and $x\in E$, let $\P_{(m,x)}$ denote the law of the time-inhomogeneous Markov chain with transition kernels $(K_n)$, starting from state $x$ at time $m$. As above, $\zeta$ denotes its lifetime. 

We define
\begin{align}
  g(x)&:=\sum_{y\in E}Q(x,y)a(x,y)\qquad \text{for}\ x\in \E,\\
  \kappa&:=\pi[g]=\sum_{y\in E}g(y)\pi(y).
\end{align}
Since $Q$ is irreducible and $\pi[g-\kappa]=0$, there exists a unique function $u:E\to \R$ satisfying the Poisson equation
\begin{align}
\label{Poisson-eq}
  (I-Q)u=g-\kappa,\qquad \pi[u]=0.
\end{align}
We also consider the following additional condition. 
\begin{itemize}
  \item There exists a function $V:E\to \R$ such that
  \begin{align}
  \label{C3}
    a(x,y)=\kappa+V(y)-V(x)\qquad \text{whenever}\ K(x,y)>0.
    \tag{C3}
  \end{align}
\end{itemize}
Under this additional assumption, some of the results in the next subsection admit sharper estimates.

\begin{Rem}
\label{Rem1.1}
For a fixed initial condition, the transition probabilities from states that are unreachable at a given time do not affect the law of the process started from that initial condition. Thus, when applying our results to concrete examples, it may be possible to construct a sequence of kernels satisfying the assumptions of this paper by appropriately specifying the rows corresponding to unreachable states.

For example, for the elephant random walk considered in Subsection \ref{S3.1} and started at the origin, we have $K_0^{\mathrm{ERW}}(0,1)=q$ and $K_0^{\mathrm{ERW}}(0,-1)=1-q$, whereas the rows $K_0^{\mathrm{ERW}}(x,\cdot)$ for $x\neq 0$ are not determined by the law of the process started from the origin. If these rows are set to zero for convenience, the resulting kernel does not have the same support as the corresponding simple random walk kernel, and hence (\ref{H1}) is not satisfied. Nevertheless, the results of this paper can be applied without changing the law of the original process by completing the rows at unreachable states as follows.

More generally, consider a time-inhomogeneous Markov chain, possibly subject to killing, started from state $x_0$ at time $0$. Define the set of states reachable at time $n$ before killing by
\begin{align}
\label{An-def}
 A_n:=\Big\{x\in E:\ \P_{(0,x_0)}(X_n=x,\ n<\zeta)>0\Big\}. 
\end{align}
Suppose that the original transition kernels $(K_n)$ satisfy
\begin{align}
K_n(x,y)>0\ \Longleftrightarrow\  K(x,y)>0\qquad \text{for}\ n\ge 0,\ x\in A_n.
\end{align}
Assume also that the expansion in (\ref{H2}) holds on the reachable rows $x\in A_n$. Increasing the integer $N$ in (\ref{H2}), if necessary, define the completed kernels by
\begin{align}
\label{tildeK}
  \widetilde{K}_n(x,y):=\begin{cases}
K_n(x,y),&\text{for}\ x\in A_n,\\
K(x,y),&\text{for}\ x\notin A_n,\  0\le n<N,\\
\displaystyle K(x,y)\left(1+\frac{a(x,y)}{n}+\frac{b(x,y)}{n^2}\right),&\text{for}\ x\notin A_n,\ n\ge N.
\end{cases}
\end{align}
If $N$ is chosen sufficiently large, $(\widetilde{K}_n)_{n\ge 0}$ is a sequence of nonzero, irreducible substochastic kernels satisfying (\ref{H1}) and (\ref{H2}) on the entire state space. Since this completion leaves the reachable rows $K_n(x,\cdot)$, $x\in A_n$, unchanged, it preserves the law of the process started from $(0,x_0)$. Similarly, for every $m\ge0$ and $x\in A_m$, the law of the process started from state $x$ at time $m$ is also preserved. Consequently, applying the main results below to $(\widetilde{K}_n)$ and restricting their conclusions to initial pairs $(m,x)$ with $x\in A_m$ yields the corresponding results for the original process.
\end{Rem}

\subsection{Main Results}
We now state our main results. 

\begin{thm}
\label{thm1.1}
There exists a function $h:\Z_{\ge0}\times E\to(0,\infty)$ such that, for each fixed $m\ge0$, each $r=0,...,d-1$, and each $x\in E_r$, 
\begin{align}
\label{thm1.1-eq2}
\P_{(m,x)}(n<\zeta)=\pi_{r+n-m}\left[\frac{1}{R}\right]\lambda^{n-m}n^{\kappa}h(m,x)\Big(1+O(n^{-1})\Big)\qquad \text{as}\ n\to \infty.
\end{align}
Here, the function $(m,x)\mapsto\lambda^{-m}h(m,x)$ is positive and space-time harmonic with respect to $(K_m)$, in the sense of (\ref{harmonic}). Moreover, for every $k\ge m\ge 0$, $x\in E$, and $\Lambda_{m,k}\in \sigma(X_m,...,X_k)$, we have
 \begin{align}
 \label{thm1.1-eq1}
   \lim_{n\to \infty}\P_{(m,x)}\Big(\Lambda_{m,k}\mid n<\zeta\Big)=\P_{(m,x)}\left[1_{\Lambda_{m,k}}\cdot \lambda^{m-k}\frac{h(k,X_k)}{h(m,x)}1_{\{k<\zeta\}}\right].
 \end{align}
\end{thm}
The function $h$ is defined by (\ref{h-def-1}) for $m\ge N$ and extended to $0\le m<N$ by the backward recursion (\ref{h-def-2}). This theorem is proved through a series of propositions in Subsection \ref{S2.1}.

We define the transformed transition kernel by
\begin{align}
\label{Qn-definition}
  Q_n(x,y):=\lambda^{-1}\frac{h(n+1,y)}{h(n,x)}K_n(x,y)\qquad \text{for}\ x,y\in E,\ n\ge 0.
\end{align}
By Theorem \ref{thm1.1}, each $Q_n$ is a stochastic kernel. For $m\ge 0$ and $x\in E$, we denote by $\Q_{(m,x)}$ the law of the time-inhomogeneous Markov chain with transition kernels $(Q_n)_{n\ge 0}$, starting from $x$ at time $m$. Under this law, $\zeta=\infty$ almost surely. By (\ref{thm1.1-eq1}), this chain can be interpreted as the original chain conditioned never to be killed.

\begin{thm}
\label{thm1.2}
All limits and asymptotic estimates below are understood as $n\to\infty$, with the initial time $m$ fixed. The following assertions hold:
\begin{enumerate}
  \item For $x,y\in E$, we have
\begin{align}
\label{thm1.2-eq1}
    Q_n(x,y)=Q(x,y)\left(1+\frac{a(x,y)-\kappa+u(y)-u(x)}{n}\right)+O(n^{-2}).
  \end{align}
  \item For $r=0,...,d-1,\ x\in E_r$, and $m\ge 0$, we have
  \begin{align}
  \label{thm1.2-eq2}
\Big\|\Q_{(m,x)}(X_n\in \cdot)-\pi_{r+n-m}\Big\|_{\mathrm{TV}}=O(n^{-1}).
\end{align}
\item For $x\in E,\ m\ge 0,$ and $f:E\to \R$, we have
\begin{align}
\label{thm1.2-eq3}
  \frac{1}{n}\sum_{k=m}^{n-1}f(X_k)\to \pi[f]\qquad \Q_{(m,x)}\text{-a.s.}
\end{align}
\end{enumerate}
If condition (\ref{C3}) holds, the conclusions in parts 1 and 2 can be sharpened as follows:
\begin{enumerate}
\item[1'.] For $x,y\in E$, we have
\begin{align}
\label{thm1.2-eq1'}
    Q_n(x,y)=Q(x,y)+O(n^{-2}).
  \end{align}
  \item[2'.] For $r=0,...,d-1,\ x\in E_r$, and $m\ge 0$, we have
  \begin{align}
  \label{thm1.2-eq2'}
\Big\|\Q_{(m,x)}(X_n\in \cdot)-\pi_{r+n-m}\Big\|_{\mathrm{TV}}=O(n^{-2}).
\end{align}
\end{enumerate}
\end{thm}
Here, $\|\cdot\|_{\mathrm{TV}}$ denotes total variation distance. This theorem is proved through a series of propositions in Subsection \ref{S2.2}.

\subsection{Organization}
The remainder of this paper is organized as follows. In Section \ref{S2}, we prove our main results. Theorem \ref{thm1.1} will be proved in Subsection \ref{S2.1}, while Theorem \ref{thm1.2} will be proved in Subsection \ref{S2.2}. In Section \ref{S3}, we apply our general theory to several concrete Markov chain models.


\section{Conditioning Problem}
\label{S2}
\subsection{Conditioning Limit}
\label{S2.1}
In this subsection, we prove Theorem \ref{thm1.1}. To this end, we establish several lemmas. For $n\ge m\ge 0$, we define the product of the one-step likelihood ratios along the path segment $(X_m,\ldots,X_n)$ by
\begin{align}
  R_{m,m}=1,\qquad R_{m,n}=\prod_{k=m}^{n-1}\frac{K_k(X_k,X_{k+1})}{K(X_k,X_{k+1})}.
\end{align}

\begin{lem}
For $n\ge m\ge 0$, $x\in E$, and $\Lambda_{m,n}\in \sigma(X_m,...,X_n)$, we have
\begin{align}
\label{Plambda=Q}
  \P_{(m,x)}\Big(\Lambda_{m,n},\ n<\zeta\Big)=\lambda^{n-m}R(x)\mathbf{Q}_{(m,x)}\left[1_{\Lambda_{m,n}}\cdot\frac{R_{m,n}}{R(X_n)}\right].
\end{align}
In particular, we have
\begin{align}
 \label{lifetime-asym} \P_{(m,x)}(n<\zeta)=\lambda^{n-m}R(x)\mathbf{Q}_{(m,x)}\left[\frac{R_{m,n}}{R(X_n)}\right].
\end{align}
\end{lem}
\begin{proof}
Set $x_m=x$ and define
\begin{align}
  A_{x_m,...,x_n}:=\Big\{X_m=x_m,\ X_{m+1}=x_{m+1},\ ...,\ X_n=x_n\Big\}\qquad \text{for}\ x_{m+1},...,x_n\in E.
\end{align}
By the definition of $Q$, we have
\begin{align*}
\label{Axm...xn}
\P_{(m,x)}(A_{x_m,...,x_n})&=\prod_{k=m}^{n-1}K_k(x_k,x_{k+1})\\
&=\prod_{k=m}^{n-1}K(x_k,x_{k+1})\cdot \frac{K_k(x_k,x_{k+1})}{K(x_k,x_{k+1})}\\
&=\lambda^{n-m}\frac{R(x)}{R(x_n)}\prod_{k=m}^{n-1}Q(x_k,x_{k+1})\prod_{k=m}^{n-1}\frac{K_k(x_k,x_{k+1})}{K(x_k,x_{k+1})}\\
&=\lambda^{n-m}R(x)\mathbf{Q}_{(m,x)}\left[ \frac{R_{m,n}}{R(X_n)},\ A_{x_m,...,x_n}\right].
\stepcounter{equation}\tag{\theequation}
\end{align*}
Up to a $\P_{(m,x)}$-null set, the event $\Lambda_{m,n}\cap\{n<\zeta\}$ is a disjoint union of cylinder events of the form $A_{x_m,\ldots,x_n}$. 
Summing (\ref{Axm...xn}) over these disjoint cylinder events and using (\ref{Q-lifetime}), we obtain the desired assertion.
\end{proof}

For $n\ge m\ge N$, we define
\begin{align}
  W_{m,n}:=\exp \left(-\kappa \sum_{k=m}^{n-1}\frac{1}{k}\right)R_{m,n}.
\end{align}

\begin{lem}
\label{lem2.2}
For $m\ge N$ and $x\in E$, we have
 \begin{align}
   W_{m,\infty}:=\lim_{n\to \infty}W_{m,n}\in (0,\infty),\qquad \mathbf{Q}_{(m,x)}\text{-a.s.}
 \end{align}
For every $p\in \R$, there exists a constant $C_{1,p}<\infty$ such that 
 \begin{align}
 \label{QW-uniform}
   \sup_{n\ge m}\mathbf{Q}_{(m,x)}[W_{m,n}^p]\le e^{\frac{C_{1,p}}{m-1}}.
 \end{align}
Moreover, as $m\to\infty$, we have
 \begin{align}
 \label{QW-asym}
   \mathbf{Q}_{(m,x)}[W_{m,\infty}]=1+O(m^{-1}).
 \end{align}
\end{lem}
\begin{proof}
We define the following quantities:
 \begin{align}
   \phi_n(x)&:=1+\frac{u(x)}{n},\\
   F_n(x,y)&:=e^{-\frac{\kappa}{n}}Q(x,y)\frac{K_n(x,y)}{K(x,y)},\\
   G_n(x,y)&:=\frac{F_n(x,y)\phi_{n+1}(y)}{F_n\phi_{n+1}(x)},\\
   M_{m,n}&:=\prod_{k=m}^{n-1}\frac{G_k(X_k,X_{k+1})}{Q(X_k,X_{k+1})}.
 \end{align}
We have
 \begin{align*}
 \label{Wmn-expand}
 W_{m,n}&=\exp \left(-\kappa \sum_{k=m}^{n-1}\frac{1}{k}\right)\prod_{k=m}^{n-1}\frac{K_k(X_k,X_{k+1})}{K(X_k,X_{k+1})}\\
 &=\frac{\phi_m(X_m)}{\phi_n(X_n)}\left(\prod_{k=m}^{n-1}\frac{\phi_{k+1}{(X_{k+1})}}{\phi_k(X_{k})}\right)\left(\prod_{k=m}^{n-1}e^{-\frac{\kappa}{k}}\frac{K_k(X_k,X_{k+1})}{K(X_k,X_{k+1})}\right)\\
 &=\frac{\phi_m(X_m)}{\phi_n(X_n)}\cdot M_{m,n}\prod_{k=m}^{n-1}\frac{F_k\phi_{k+1}(X_k)}{\phi_k(X_k)}.
\stepcounter{equation}\tag{\theequation}
 \end{align*}
We examine each factor on the right-hand side of (\ref{Wmn-expand}).
\begin{itemize}
  \item (The first factor) Since $\phi_n(X_n)\to 1$, the first factor of (\ref{Wmn-expand}) converges to $\phi_m(X_m)$.
  \item (The third factor) By assumption (\ref{H2}), we have
  \begin{align*}
  F_n(x,y)-Q(x,y)&=Q(x,y)(e^{-\frac{\kappa}{n}}-1)+Q(x,y)\frac{e^{-\frac{\kappa}{n}}a(x,y)}{n}+O(n^{-2})\\
 &=\frac{a(x,y)Q(x,y)-\kappa Q(x,y)}{n}+O(n^{-2}).
 \stepcounter{equation}\tag{\theequation}
  \end{align*}
 Thus, we have
  \begin{align}
  \label{Fn-expand}
    F_n(x,y)=Q(x,y)+\frac{a(x,y)Q(x,y)-\kappa Q(x,y)}{n}+O(n^{-2}).
  \end{align}
By (\ref{Poisson-eq}), we have
   \begin{align*}
   \label{Fnphi-n+1}
  F_n\phi_{n+1}(x)&=\sum_{y\in E}F_n(x,y)\left(1+\frac{u(y)}{n+1}\right)\\
  &=1+\frac{g(x)-\kappa}{n}+\frac{Qu(x)}{n}+O(n^{-2})\\
  &=\phi_n(x)+O(n^{-2}).
  \stepcounter{equation}\tag{\theequation}
  \end{align*}
 Since there exists a constant $C_2<\infty$ such that
  \begin{align*}
    \sum_{k=m}^\infty \left|\log \left(\frac{F_k\phi_{k+1}(X_k)}{\phi_k(X_k)}\right)\right|&=\sum_{k=m}^\infty \Big|\log \Big(1+O(k^{-2})\Big)\Big|=\sum_{k=m}^\infty O(k^{-2})\le \frac{C_2}{m},
    \stepcounter{equation}\tag{\theequation}
  \end{align*}
the infinite product in the third factor of (\ref{Wmn-expand}) converges absolutely, and
  \begin{align}
  \label{prodFphi-asy}
   \prod_{k=m}^\infty \frac{F_k\phi_{k+1}(X_k)}{\phi_k(X_k)}=1+O(m^{-1}).
  \end{align}
  \item (The second factor) For $n\ge N$, each $G_n$ is a stochastic transition kernel. Thus, by the Markov property, $(M_{m,n})_{n\ge m}$ is a positive $\mathbf{Q}_{(m,x)}$-martingale with mean one. The martingale convergence theorem implies that there exists $M_{m,\infty}\ge 0$ such that
  \begin{align}
  \label{Mmn-lim}
    M_{m,n}\to M_{m,\infty},\qquad \mathbf{Q}_{(m,x)}\text{-a.s.}
  \end{align}
By (\ref{Fn-expand}), there exist functions $\e_k:E\times E\to \R$ and a constant $C_3<\infty$ such that
  \begin{align*}
  \frac{G_k(x,y)}{Q(x,y)}=1+\e_k(x,y),\qquad |\e_k(x,y)|\le \frac{C_3}{k}.
\stepcounter{equation}\tag{\theequation} 
  \end{align*}
Moreover, we have
  \begin{align}
    \sum_{y\in E}Q(x,y)\e_k(x,y)=\sum_{y\in E}G_k(x,y)-\sum_{y\in E}Q(x,y)=0.
  \end{align}
Thus, for each $p\in\mathbb R$, Taylor's theorem yields constants $C_{4,p},C_{5,p}<\infty$ such that
  \begin{align*}
  \label{Taylor-exp}
    \sum_{y\in E}Q(x,y)\Big(1+\e_k(x,y)\Big)^p&\le \sum_{y\in E}Q(x,y)\Big(1+p\e_k (x,y)+C_{4,p}\e_k(x,y)^2\Big)\\
    &\le 1+\frac{C_{5,p}}{k^2}.
    \stepcounter{equation}\tag{\theequation} 
  \end{align*}
Hence, iterating conditional expectations and using (\ref{Taylor-exp}), we obtain
  \begin{align*}
  \label{QMp}
\mathbf{Q}_{(m,x)}[M_{m,n}^p]\le \prod_{k=m}^{n-1}\left(1+\frac{C_{5,p}}{k^2}\right)\le e^{\frac{C_{5,p}}{m-1}}.
\stepcounter{equation}\tag{\theequation}
  \end{align*}
In particular, setting $p=-1$ and applying Fatou's lemma, we have
  \begin{align}
    \mathbf{Q}_{(m,x)}[M_{m,\infty}^{-1}]\le \varliminf_{n\to \infty}\mathbf{Q}_{(m,x)}[M_{m,n}^{-1}]\le e^{\frac{C_{5,-1}}{m-1}}<\infty.
  \end{align}
This implies $M_{m,\infty}>0,\ \mathbf{Q}_{(m,x)}$-a.s. Finally, taking $p=2$ in (\ref{QMp}) shows that $(M_{m,n})$ is bounded in $L^2$ and hence uniformly integrable. Thus, we have
\begin{align}
  M_{m,n}\to M_{m,\infty}\qquad \text{in}\ L^1(\mathbf{Q}_{(m,x)}).
\end{align}
Therefore, we have
\begin{align}
\label{QM-lim}
  \mathbf{Q}_{(m,x)}[M_{m,\infty}]=\lim_{n\to \infty}\mathbf{Q}_{(m,x)}[M_{m,n}]=1.
\end{align}
\end{itemize}
By (\ref{Wmn-expand}), (\ref{prodFphi-asy}), and (\ref{Mmn-lim}), we have
\begin{align}
W_{m,\infty}
=\phi_m(x)M_{m,\infty}\Big(1+O(m^{-1})\Big).
\end{align}
Taking expectations and using (\ref{QM-lim}), we obtain
\begin{align}
  \mathbf Q_{(m,x)}[W_{m,\infty}]
=\phi_m(x)\Big(1+O(m^{-1})\Big)
=1+O(m^{-1}).
\end{align}
This completes the proof.
\end{proof}

Under condition (\ref{C3}), we obtain the following pathwise estimates:

\begin{cor}
Assume that condition (\ref{C3}) holds. There exists a constant $C_6<\infty$ such that for $m\ge N$ and $x\in E$, 
\begin{align}
  e^{-\frac{C_6}{m}}\le W_{m,\infty}\le e^{\frac{C_6}{m}},\qquad \mathbf{Q}_{(m,x)}\text{-a.s.}
\end{align}
Moreover, for $n\ge m\ge N$ we have
\begin{align}
  \Big|W_{m,\infty}-W_{m,n}\Big|\le \frac{C_6e^{\frac{C_6}{m}}}{n},\qquad \mathbf{Q}_{(m,x)}\text{-a.s.}
\end{align}
\end{cor}
\begin{proof}
By (\ref{H2}) and (\ref{C3}), we have
\begin{align*}
\log W_{m,n}&=-\kappa \sum_{k=m}^{n-1}\frac{1}{k}+\sum_{k=m}^{n-1}\log \left(1+\frac{\kappa+V(X_{k+1})-V(X_k)}{k}+\frac{b(X_k,X_{k+1})}{k^2}+O(k^{-3})\right)\\
&=\sum_{k=m}^{n-1}\frac{V(X_{k+1})-V(X_k)}{k}+\sum_{k=m}^{n-1}\frac{b(X_k,X_{k+1})-\frac{1}{2}a(X_k,X_{k+1})^2}{k^2}+O\left(\sum_{k=m}^{n-1}k^{-3}\right).
\stepcounter{equation}\tag{\theequation}
\end{align*}
We have
\begin{align*}
\left|\sum_{k=m}^{n-1}\frac{V(X_{k+1})-V(X_k)}{k}\right|&=\left|-\frac{V(X_m)}{m}+\frac{V(X_n)}{n}+\sum_{k=m}^{n-1}\frac{V(X_{k+1})}{k(k+1)}\right|\\
&\le \|V\|_\infty\left(\frac{1}{m}+\frac{1}{n}+\sum_{k=m}^{n-1}\frac{1}{k(k+1)}\right)\\
&\le \frac{2\|V\|_\infty}{m}.
\stepcounter{equation}\tag{\theequation}
\end{align*}
Since $a$ and $b$ are bounded, there exists a constant $C_6<\infty$ such that
\begin{align}
  \Big|\log W_{m,n}\Big|\le \frac{2\|V\|_\infty}{m}+\left(\|b\|_\infty+\frac{\|a\|_\infty^2}{2}\right)\sum_{k=m}^{n-1}\frac{1}{k^2}+O\left(\sum_{k=m}^{n-1}k^{-3}\right)\le \frac{C_6}{m}.
\end{align}
Letting $n\to\infty$, we obtain the first assertion.

Next, we prove the second assertion. Since $W_{m,l}=W_{m,n}W_{n,l}$ for $m\le n\le l$, we have
\begin{align}
  \Big|\log W_{m,l}-\log W_{m,n}\Big|=\Big|\log W_{n,l}\Big|\le\frac{C_6}{n}.
\end{align}
Letting $l\to \infty$, we have
\begin{align}
    \Big|\log W_{m,\infty}-\log W_{m,n}\Big|=\Big|\log W_{n,\infty}\Big|\le\frac{C_6}{n}.
\end{align}
Thus, by the mean value theorem, we have
\begin{align}
  \Big|W_{m,\infty}-W_{m,n}\Big|=\Big|e^{\log W_{m,\infty}}-e^{\log W_{m,n}}\Big|\le e^{\frac{C_6}{m}} \Big|\log W_{m,\infty}-\log W_{m,n}\Big|\le\frac{C_6e^{\frac{C_6}{m}}}{n}.
\end{align}
This completes the proof.
\end{proof}

Let $H_m$ denote the $m$-th harmonic number:
\begin{align}
  H_0:=0,\qquad H_m:=\sum_{k=1}^m\frac{1}{k}.
\end{align}
Let $\gamma$ denote Euler's constant. For $m\ge N$ and $x\in E$, we define
\begin{align}
\label{h-def-1}
  h(m,x):=e^{\kappa(\gamma-H_{m-1})}R(x)\mathbf{Q}_{(m,x)}[W_{m,\infty}].
\end{align}
By Lemma \ref{lem2.2}, we have $0<h<\infty$.

\begin{prop}
\label{lem2.3}
For $m\ge N$ and $x\in E$, we have
 \begin{align}
 \label{harmonic}
  \sum_{y\in E} K_m(x,y)h(m+1,y)=\lambda h(m,x).
 \end{align}
In particular, the function $(m,x)\mapsto \lambda^{-m}h(m,x)$ is a space-time harmonic function for $(K_m).$
\end{prop}
\begin{proof}
By the Markov property under $\mathbf{Q}_{(m,x)}$, we have
\begin{align*}
  \sum_{y\in E} K_m(x,y)h(m+1,y)&=e^{\kappa(\gamma-H_m)}\sum_{y\in E}K_m(x,y)R(y)\mathbf{Q}_{(m+1,y)}[W_{m+1,\infty}]\\
  &=e^{\kappa(\gamma-H_{m-1})}e^{-\frac{\kappa}{m}}\sum_{y\in E} R(y)K(x,y)\cdot \frac{K_m(x,y)}{K(x,y)}\mathbf{Q}_{(m+1,y)}[W_{m+1,\infty}]\\
  &=\lambda R(x)e^{\kappa(\gamma-H_{m-1})}e^{-\frac{\kappa}{m}}\sum_{y\in E} Q(x,y) \frac{K_m(x,y)}{K(x,y)}\mathbf{Q}_{(m+1,y)}[W_{m+1,\infty}]\\
  &=\lambda R(x)e^{\kappa(\gamma-H_{m-1})}\mathbf{Q}_{(m,x)}[W_{m,\infty}]\\
  &=\lambda h(m,x).
\stepcounter{equation}\tag{\theequation}
\end{align*}  
This completes the proof.
\end{proof}

\begin{lem}
There exist constants $C_7<\infty$ and $\eta\in (0,1)$ such that for any $f:E\to \R,\ n\ge m\ge 0,\ r=0,...,d-1$, and $x\in E_r$, we have
\begin{align}
\label{Qpi-speed}
  \Big|\mathbf{Q}_{(m,x)}[f(X_n)]-\pi_{r+n-m}[f]\Big|\le C_7\|f\|_\infty\eta^{n-m}.
\end{align}
\end{lem}
\begin{proof}
By a standard property of the cyclic decomposition (see, e.g., Theorem 18.4 of \cite{Klenke}), $P_r:=Q^d|_{E_r\times E_r}$ is an irreducible and aperiodic transition kernel on $E_r.$ Its invariant distribution is $\pi_r.$ By Section 18.4 of \cite{Klenke}, there exist constants $C_{8,r}<\infty$ and $\alpha_r\in (0,1)$ such that for every function $g:E_r\to \R$,
\begin{align}
  \Big|P_r^kg(x)-\pi_r[g]\Big|\le C_{8,r}\|g\|_\infty\alpha_r^k\qquad \text{for}\ k\ge 0,\ x\in E_r.
\end{align}
Set
\begin{align}
  C_9:=\max_{0\le r\le d-1}C_{8,r}<\infty,\qquad \alpha:=\max_{0\le r\le d-1}\alpha_r<1.
\end{align}
Let $n-m=kd+s$ with $k\ge 0$ and $0\le s\le d-1$. Since
\begin{align}
  \pi_{r+n-m}=\pi_{r+s}=\pi_rQ^s,
\end{align}
we have
\begin{align*}
  \Big|\mathbf{Q}_{(m,x)}[f(X_n)]-\pi_{r+n-m}[f]\Big|&= \Big|P_r^k(Q^sf)(x)-\pi_r[Q^sf]\Big|\\
  &\le C_9\|Q^s f\|_\infty\alpha^k\\
  &\le C_9\|f\|_\infty\alpha^k.
  \stepcounter{equation}\tag{\theequation}
\end{align*}
The assertion follows by setting $\eta=\alpha^{\frac{1}{d}}\in(0,1)$ and $C_7=C_9\eta^{-(d-1)}$. This completes the proof.
\end{proof}

We now prove (\ref{thm1.1-eq2}) for $m\ge N$.

\begin{prop}
 For $m\ge N$ and $x\in E_r$, we have
 \begin{align}
 \label{lifetime-limit}
   \P_{(m,x)}(n<\zeta)=\pi_{r+n-m}\left[\frac{1}{R}\right]\lambda^{n-m}n^{\kappa}h(m,x)\Big(1+O(n^{-1})\Big).
 \end{align}
\end{prop}
\begin{proof}
Define a probability measure $\widetilde{\mathbf{Q}}_{(m,x)}$ by
\begin{align}
\label{Qtilde-Q-infty}
  \frac{d\widetilde{\mathbf{Q}}_{(m,x)}}{d\mathbf{Q}_{(m,x)}}=\frac{W_{m,\infty}}{\mathbf{Q}_{(m,x)}[W_{m,\infty}]}\qquad \text{on}\ \sigma(X_m,X_{m+1},...).
\end{align}
By the Markov property and $W_{m,\infty}=W_{m,n}W_{n,\infty}$, we have
\begin{align}
  \label{Qtilde-Q-mn} \left.\frac{d\widetilde{\mathbf{Q}}_{(m,x)}}{d\mathbf{Q}_{(m,x)}}\right|_{\sigma(X_m,...,X_n)}=W_{m,n}\frac{\mathbf{Q}_{(n,X_n)}[W_{n,\infty}]}{\mathbf{Q}_{(m,x)}[W_{m,\infty}]}.
\end{align}
Under $\widetilde{\mathbf{Q}}_{(m,x)}$, the process $(X_k)_{k\ge m}$ is a time-inhomogeneous Markov chain with transition kernels
\begin{align}
\label{tildeQ-trans}
  \widetilde{Q}_k(x,y):=e^{-\frac{\kappa}{k}}Q(x,y)\frac{K_k(x,y)}{K(x,y)}\frac{\mathbf{Q}_{(k+1,y)}[W_{k+1,\infty}]}{\mathbf{Q}_{(k,x)}[W_{k,\infty}]}.
\end{align}
Thus, by (\ref{H2}) and (\ref{QW-asym}), we have
\begin{align}
\label{tildeQk-Q-difference}
  \widetilde{Q}_k(x,y)=Q(x,y)+O(k^{-1}).
\end{align}

Since $m$ is fixed and $n\to\infty$, we may assume that $n>2m$. Let $\mu_n:=\widetilde{\mathbf{Q}}_{(m,x)}(X_n\in \cdot)$ and $j:=[\frac{n}{2}]$. We have
\begin{align*}
\label{mun-pi-expand}
  \mu_n-\pi_{r+n-m}&=\mu_jQ^{n-j}+\sum_{k=j}^{n-1}\Big(\mu_{k+1}Q^{n-1-k}-\mu_k Q^{n-k}\Big)-\pi_{r+j-m}Q^{n-j}\\
  &=(\mu_j-\pi_{r+j-m})Q^{n-j}+\sum_{k=j}^{n-1}\Big(\mu_{k+1}-\mu_k Q\Big)Q^{n-1-k}.
  \stepcounter{equation}\tag{\theequation}
\end{align*}
By (\ref{Qpi-speed}), the total variation norm of the first term on the right-hand side of (\ref{mun-pi-expand}) is bounded as follows:
\begin{align*}
&\Big\|(\mu_j-\pi_{r+j-m})Q^{n-j}\Big\|_{\text{TV}}\\
&\qquad =\sup_{A\subset E}\left|\sum_{x\in E_{r+j-m}}\Big(\mu_j(x)-\pi_{r+j-m}(x)\Big)Q^{n-j}(x,A)\right|\\
&\qquad =\sup_{A\subset E}\left|\sum_{x\in E_{r+j-m}}\Big(\mu_j(x)-\pi_{r+j-m}(x)\Big)\Big(Q^{n-j}(x,A)-\pi_{r+n-m}(A)\Big)\right|\\
&\qquad\le  2C_7 \eta^{n-j}.
\stepcounter{equation}\tag{\theequation}
\end{align*}
By (\ref{Qpi-speed}) and (\ref{tildeQk-Q-difference}), there exists a constant $C_{10}<\infty$ such that each summand in the second term on the right-hand side of (\ref{mun-pi-expand}) satisfies
\begin{align*}
\label{mu-muQ}
  &\Big\|\Big(\mu_{k+1}-\mu_k Q\Big)Q^{n-1-k}\Big\|_{\text{TV}}\\
  &\qquad =\sup_{A\subset E}\left|\sum_{y\in E}\Big(\mu_{k+1}(y)-\mu_kQ(y)\Big)\Big(Q^{n-1-k}(y,A)-\pi_{r+n-m} (A)\Big)\right|\\
  &\qquad \le C_7 \eta^{n-1-k}\sum_{y\in E}\mu_k(y)\Big\|\widetilde{Q}_k(y,\cdot)-Q(y,\cdot)\Big\|_{\text{TV}}\\
  &\qquad \le \frac{C_{10}\eta^{n-1-k}}{k}.
\stepcounter{equation}\tag{\theequation}
\end{align*}
Thus, we have
\begin{align}
\label{mun-pi}
  \Big\|\mu_n-\pi_{r+n-m}\Big\|_{\text{TV}}\le 2C_7\eta^{n-j}+C_{10}\sum_{k=j}^{n-1}\frac{\eta^{n-1-k}}{k}=O(j^{-1})=O(n^{-1}).
\end{align}
Therefore, combining (\ref{Qtilde-Q-mn}) and (\ref{QW-asym}), for every function $f:E\to \R$, we have
\begin{align*}
\label{QWf-asym}
  \mathbf{Q}_{(m,x)}[W_{m,n}f(X_n)]&=\mathbf{Q}_{(m,x)}[W_{m,\infty}]\widetilde{\mathbf{Q}}_{(m,x)}\left[\frac{f(X_n)}{\mathbf{Q}_{(n,X_n)}[W_{n,\infty}]}\right]\\
  &=\mathbf{Q}_{(m,x)}[W_{m,\infty}]\Big(\pi_{r+n-m}[f]+O(n^{-1})\|f\|_\infty\Big).
  \stepcounter{equation}\tag{\theequation}
\end{align*}
We have
\begin{align*}
\label{expkappa-asym}
  \exp\left(\kappa\sum_{k=m}^{n-1}\frac{1}{k}\right)&=\exp\Big(\kappa\Big(\log n+\gamma -H_{m-1}+O(n^{-1})\Big)\Big)\\
  &=n^\kappa e^{\kappa(\gamma-H_{m-1})}\Big(1+O(n^{-1})\Big).
\stepcounter{equation}\tag{\theequation}
\end{align*}
Taking $f=\frac{1}{R}$ in (\ref{QWf-asym}) and combining the resulting estimate with (\ref{lifetime-asym}) and (\ref{expkappa-asym}), we obtain the desired assertion.
\end{proof}

We now prove (\ref{thm1.1-eq1}) for the case $m\ge N$.

\begin{prop}
\label{prop2.7}
For $k\ge m\ge N,\ x\in E$, and $\ \Lambda_{m,k}\in\sigma(X_m,...,X_k)$, we have
 \begin{align}
 \label{Conditioning}
   \lim_{n\to \infty}\P_{(m,x)}\Big(\Lambda_{m,k}\mid n<\zeta\Big)=\P_{(m,x)}\left[1_{\Lambda_{m,k}}\cdot \lambda^{m-k}\frac{h(k,X_k)}{h(m,x)}1_{\{k<\zeta\}}\right].
 \end{align}
\end{prop}
\begin{proof}
  By the Markov property and (\ref{lifetime-limit}), we have
  \begin{align*}
  \P_{(m,x)}\Big(\Lambda_{m,k}\mid n<\zeta\Big)&=\P_{(m,x)}\left[1_{\Lambda_{m,k}}1_{\{k<\zeta\}}\cdot \frac{\P_{(k,X_k)}(n<\zeta)}{\P_{(m,x)}(n<\zeta)}\right]\\
  &=\sum_{y\in E}\frac{\P_{(k,y)}(n<\zeta)}{\P_{(m,x)}(n<\zeta)}\cdot \P_{(m,x)}\Big[1_{\Lambda_{m,k}}1_{\{k<\zeta\}},\ X_k=y\Big]\\
  &\to \sum_{y\in E}\lambda^{m-k}\frac{h(k,y)}{h(m,x)}\cdot \P_{(m,x)}\Big[1_{\Lambda_{m,k}}1_{\{k<\zeta\}},\ X_k=y\Big]\\
  &=\P_{(m,x)}\left[1_{\Lambda_{m,k}}\cdot \lambda^{m-k}\frac{h(k,X_k)}{h(m,x)}1_{\{k<\zeta\}}\right].
\stepcounter{equation}\tag{\theequation}   
  \end{align*}
  This completes the proof.
\end{proof}

The preceding results were proved for $m\ge N$. To extend them to $0\le m<N$, define $h(m,x)$ by the backward recursion
\begin{align}
\label{h-def-2}
  h(m,x):=\lambda^{-1}\sum_{y\in E}K_m(x,y)h(m+1,y)\qquad \text{for}\ m=N-1,N-2,...,0.
\end{align}
Note that $0<h(m,x)<\infty$. The harmonicity relation (\ref{harmonic}) holds by construction. Applying the Markov property at time $N$ extends the survival probability asymptotics in (\ref{lifetime-limit}) to $0\le m<N$. The conditional limit (\ref{Conditioning}) then follows by the same argument as in Proposition \ref{prop2.7}. Consequently, for every $m\ge0$, the conditioned law admits the following h-transform representation:
\begin{align}
  \left.\frac{d\Q_{(m,x)}}{d\P_{(m,x)}}\right|_{\sigma(X_m,...,X_n)}=\lambda^{m-n}\frac{h(n,X_n)}{h(m,x)}1_{\{n<\zeta\}}\qquad \text{for}\ 0\le m\le n.
\end{align}
This is equivalent to (\ref{Qn-definition}).

This completes the proof of Theorem \ref{thm1.1}.

\subsection{Conditioned Process}
\label{S2.2}
In this subsection, we prove Theorem \ref{thm1.2}. We first show that the measures $\Q_{(m,x)}$ and $\mathbf{Q}_{(m,x)}$ are equivalent.

\begin{prop}
\label{prop2.8}
For $m\ge N$ and $x\in E$, we have
\begin{align}
\label{density}
  \frac{d\Q_{(m,x)}}{d\mathbf{Q}_{(m,x)}}=\frac{W_{m,\infty}}{\mathbf{Q}_{(m,x)}[W_{m,\infty}]}\qquad \text{on}\ \sigma(X_m,X_{m+1},...).
\end{align}
For $0\le m< N$, we define
\begin{align}
  W_{m,\infty}:=R_{m,N}W_{N,\infty}.
\end{align}
Then, (\ref{density}) also holds for $0\le m<N$. In particular, for $m\ge 0$ and $x\in E$, the measures $\Q_{(m,x)}$ and $\mathbf{Q}_{(m,x)}$ are equivalent.
\end{prop}
\begin{proof}
First, suppose that $m\ge N$. By (\ref{Plambda=Q}), for $n\ge m$ and $\Lambda_{m,n}\in\sigma(X_m,\ldots,X_n)$, we have
\begin{align*}
\Q_{(m,x)}(\Lambda_{m,n})&=\lambda^{n-m}R(x)\mathbf{Q}_{(m,x)}\left[1_{\Lambda_{m,n}}\cdot \frac{R_{m,n}}{R(X_n)}\cdot \lambda^{m-n}\frac{h(n,X_n)}{h(m,x)}\right]\\
&=\mathbf{Q}_{(m,x)}\left[1_{\Lambda_{m,n}}\cdot W_{m,n}\frac{\mathbf{Q}_{(n,X_n)}[W_{n,\infty}]}{\mathbf{Q}_{(m,x)}[W_{m,\infty}]}\right].
\stepcounter{equation}\tag{\theequation}
\end{align*}
This implies
\begin{align}
  \left.\frac{d\Q_{(m,x)}}{d\mathbf{Q}_{(m,x)}}\right|_{\sigma(X_m,...,X_n)}=W_{m,n}\frac{\mathbf{Q}_{(n,X_n)}[W_{n,\infty}]}{\mathbf{Q}_{(m,x)}[W_{m,\infty}]}.
\end{align}
Since this agrees with (\ref{Qtilde-Q-mn}) for every $n\ge m$, the measures $\mathbb Q_{(m,x)}$ and $\widetilde{\mathbf{Q}}_{(m,x)}$ coincide on $\sigma(X_m,X_{m+1},...)$. Thus, by (\ref{Qtilde-Q-infty}), the first assertion holds. 

In the case of $0\le m<N$, the assertion follows by a direct calculation. Finally, Lemma \ref{lem2.2} and (\ref{H1}) imply that $W_{m,\infty}>0$, $\mathbf{Q}_{(m,x)}$-a.s., for every $m\ge0$. Hence the two measures are equivalent. This completes the proof.
\end{proof}

For $f:E\to \R,$ we define
\begin{align}
  Tf(x)&:=\sum_{y\in E}Q(x,y)a(x,y)f(y),\\
  Uf(x)&:=\sum_{y\in E}Q(x,y)b(x,y)f(y).
\end{align}

We next refine the asymptotic estimate in (\ref{QW-asym}). 

\begin{lem}
For $x\in E$, we have
\begin{align}
\label{QW-asym-ex}
  \mathbf{Q}_{(m,x)}[W_{m,\infty}]=1+\frac{u(x)+\alpha}{m}+O(m^{-2}),
\end{align}
where 
\begin{align}
  \alpha:=\pi[Tu]+\pi[U1]-\frac{\kappa^2}{2}.
\end{align}
\end{lem}
\begin{proof}
We refine the approximate harmonicity estimate in (\ref{Fnphi-n+1}). By (\ref{H2}), a second-order expansion gives
  \begin{align}
    F_m=Q+\frac{D_1}{m}+\frac{D_2}{m^2}+O(m^{-3}),
  \end{align}
  where 
  \begin{align}
    D_1:=T-\kappa Q,\qquad D_2:=U-\kappa T+\frac{\kappa^2}{2}Q.
  \end{align}
  We define
  \begin{align}
    \psi_m(x):=1+\frac{A(x)}{m}+\frac{B(x)}{m^2},
  \end{align}
where $A,B:E\to\mathbb R$ are functions to be chosen. We have
  \begin{align*}
    &F_m\psi_{m+1}(x)-\psi_m(x)\\
    &\qquad =\frac{QA(x)+D_11(x)-A(x)}{m}+\frac{QB(x)-QA(x)+D_1A(x)+D_21(x)-B(x)}{m^2}+O(m^{-3}).
    \stepcounter{equation}\tag{\theequation}
  \end{align*}
Thus, for $F_m\psi_{m+1}-\psi_m=O(m^{-3})$ to hold, it suffices that $A$ and $B$ satisfy the following conditions:
\begin{align}
\label{equations1}
  QA+D_11-A=0&\qquad \Longleftrightarrow\qquad (I-Q)A=g-\kappa,\\
  \label{equations2}
  QB-QA+D_1A+D_21-B=0&\qquad \Longleftrightarrow\qquad (I-Q)B=-QA+D_1A+D_21.
\end{align}
By (\ref{Poisson-eq}), equation (\ref{equations1}) is satisfied by $A=u+c$, where $c$ is a constant to be determined. Since $Q$ is irreducible, the Poisson equation (\ref{equations2}) is solvable if and only if its right-hand side has zero $\pi$-mean. For $A=u+c$, this condition reads 
\begin{align}
  0=\pi\Big[-QA+D_1A+D_21\Big]=-c+ \pi[Tu]+\pi[U1]-\frac{\kappa^2}{2}.
\end{align}
Thus, choosing $c=\alpha$ ensures that equation (\ref{equations2}) has a solution. Therefore, for such $\psi_m$, we have
\begin{align}
  \epsilon_m:=F_m\psi_{m+1}-\psi_m=O(m^{-3}).
\end{align}
For $n>m\ge N$, we have
\begin{align*}
 \delta_m(x):&=\mathbf{Q}_{(m,x)}[W_{m,\infty}]-\psi_m(x)\\
 &=e^{-\frac{\kappa}{m}}\sum_{y\in E}Q(x,y)\frac{K_m(x,y)}{K(x,y)}\mathbf{Q}_{(m+1,y)}[W_{m+1,\infty}]-F_m\psi_{m+1}(x)+F_m\psi_{m+1}(x)-\psi_m(x)\\
 &=F_m\delta_{m+1}(x)+\epsilon_m(x)\\
 &=F_m\cdots F_{n-1}\delta_{n}(x)+\sum_{k=m}^{n-2} F_m\cdots F_k\epsilon_{k+1}(x)+\epsilon_m(x)\\
 &=\mathbf{Q}_{(m,x)}\Big[W_{m,n}\delta_n(X_n)\Big]+\sum_{k=m}^{n-1}\mathbf{Q}_{(m,x)}\Big[W_{m,k}\epsilon_k(X_k)\Big].
\stepcounter{equation}\tag{\theequation}
\end{align*}
By (\ref{QW-asym}), we have $\delta_m=O(m^{-1})$. Thus, by (\ref{QW-uniform}), there exist constants $C_{11},C_{12}<\infty$ such that
\begin{align*}
  \|\delta_m\|_\infty\le \left(\sup_{\substack{j\ge m\\x\in E}}\mathbf{Q}_{(m,x)}[W_{m,j}]\right)\left(\|\delta_n\|_\infty+\sum_{k=m}^{n-1}\|\epsilon_k\|_\infty\right)\le\frac{C_{11}}{n}+C_{12}\sum_{k=m}^{n-1}\frac{1}{k^3}.
\stepcounter{equation}\tag{\theequation}
\end{align*}
Letting $n\to \infty,$ we have
\begin{align}
  \|\delta_m\|_\infty \le C_{12}\sum_{k=m}^\infty \frac{1}{k^3}=O(m^{-2}).
\end{align}
This implies the assertion.
\end{proof}

We now prove (\ref{thm1.2-eq1}).

\begin{prop}
  For $x,y\in E$, we have
  \begin{align}
  \label{Qn=Q+O}
    Q_n(x,y)=Q(x,y)\left(1+\frac{a(x,y)-\kappa+u(y)-u(x)}{n}\right)+O(n^{-2}).
  \end{align}
\end{prop}
\begin{proof}
    By (\ref{QW-asym-ex}), we have
  \begin{align}
    \frac{\mathbf{Q}_{(n+1,y)}[W_{n+1,\infty}]}{\mathbf{Q}_{(n,x)}[W_{n,\infty}]}=1+\frac{u(y)-u(x)}{n}+O(n^{-2}).
  \end{align}
  Moreover, we have
  \begin{align}
    e^{-\frac{\kappa}{n}}=1-\frac{\kappa}{n}+O(n^{-2}).
  \end{align}
Therefore, substituting these and  (\ref{H2}) into equation (\ref{tildeQ-trans}), we obtain the desired assertion.
\end{proof}

We now prove (\ref{thm1.2-eq1'}).

\begin{cor}
\label{cor2.10}
The following statements are equivalent:
\begin{enumerate}
  \item $Q_n=Q+O(n^{-2})$.
  \item Condition (\ref{C3}) holds.
\end{enumerate}
\end{cor}
\begin{proof}
By (\ref{Qn=Q+O}), assertion 1 implies assertion 2 with $V=-u$. Conversely, suppose that assertion 2 holds. Then
\begin{align}
  (I-Q)V(x)=\sum_{y\in E}Q(x,y)\Big(V(x)-V(y)\Big)=\kappa-g(x).
\end{align}
Therefore, by the uniqueness of the solution to the Poisson equation (\ref{Poisson-eq}), we have
\begin{align}
  u(x)=-V(x)+\pi[V].
\end{align}
Substituting this into (\ref{Qn=Q+O}) yields assertion 1. This completes the proof.
\end{proof}

We now prove (\ref{thm1.2-eq2}) and (\ref{thm1.2-eq2'}).

\begin{prop}
For $r=0,...,d-1$, $x\in E_r$, and $m\ge 0$, we have
\begin{align}
\Big\|\Q_{(m,x)}(X_n\in \cdot)-\pi_{r+n-m}\Big\|_{\mathrm{TV}}=O(n^{-1}).
\end{align}
Moreover, under condition (\ref{C3}), we have
\begin{align}
\Big\|\Q_{(m,x)}(X_n\in \cdot)-\pi_{r+n-m}\Big\|_{\mathrm{TV}}=O(n^{-2}).
\end{align}
\end{prop}
\begin{proof}
First suppose that $m\ge N$. The first assertion follows from (\ref{mun-pi}). For the second assertion, Corollary \ref{cor2.10} allows us to refine the estimate in (\ref{mu-muQ}) to order $k^{-2}$. Applying this refined estimate to (\ref{mun-pi}) then yields the second assertion.

For $0\le m<N$, both assertions follow by conditioning on $X_N$, since $X_N\in E_{r+N-m}$ almost surely under $\Q_{(m,x)}$. This completes the proof.
\end{proof}

We now prove (\ref{thm1.2-eq3}).

\begin{prop}
For $x\in E,\ m\ge 0,$ and $f:E\to \R$, we have
\begin{align}
  \frac{1}{n}\sum_{k=m}^{n-1}f(X_k)\to \pi[f]\qquad \Q_{(m,x)}\text{-a.s.}
\end{align}
In particular, for $z\in E$, we have
\begin{align}
  \frac{1}{n}\#\Big\{m\le k<n:\ X_k=z\Big\}\to \pi(z)\qquad \Q_{(m,x)}\text{-a.s.}
\end{align}
\end{prop}
\begin{proof}
Under $\mathbf{Q}_{(m,x)}$, the process $(X_k)_{k\ge m}$ is an irreducible time-homogeneous Markov chain on the finite state space $E$ with invariant distribution $\pi$. Hence, the ergodic theorem (see, e.g., Example 20.17 of \cite{Klenke}) yields
 \begin{align}
  \frac{1}{n}\sum_{k=m}^{n-1}f(X_k)\to \pi[f]\qquad \mathbf{Q}_{(m,x)}\text{-a.s.}
\end{align}
The assertion now follows from the equivalence of $\Q_{(m,x)}$ and $\mathbf{Q}_{(m,x)}$ established in Proposition \ref{prop2.8}. This completes the proof.
\end{proof}

This completes the proof of Theorem \ref{thm1.2}.


\section{Examples}
\label{S3}
\subsection{Elephant Random Walk Conditioned to Stay in an Interval}
\label{S3.1}
The elephant random walk (ERW), introduced by Sch\"{u}tz and Trimper \cite{Schutz-Trimper}, is a random walk with memory. Let
\begin{align}
  X_0:=0,\qquad X_n:=\xi_1+\cdots +\xi_n.
\end{align}
The process $(X_n)_{n\ge 0}$ is called an ERW if its increments $(\xi_n)_{n\ge1}$ are generated as follows:
\begin{align}
  \xi_1&=\begin{cases}
    +1&\text{with probability}\ q,\\
    -1&\text{with probability}\ 1-q,\\
  \end{cases}
 \\
 \label{ERW-eq1}
 \xi_{n+1}&=\begin{cases}
    +\xi_{U_n}&\text{with probability}\ p,\\
    -\xi_{U_n}&\text{with probability}\ 1-p,\\
  \end{cases}
\end{align}
where $p,q\in (0,1)$ and $U_n$ is uniformly distributed on $\{1,...,n\}$. The variables $(U_n)_{n\ge1}$, the Bernoulli variables determining the choices in (\ref{ERW-eq1}), and the initial increment $\xi_1$ are mutually independent. It is known that the ERW is a time-inhomogeneous Markov chain on $\Z$. Its transition kernel is given by
\begin{align}
  K_n^{\mathrm{ERW}}(x,x\pm 1)=\frac{1}{2}\pm \left(p-\frac{1}{2}\right)\frac{x}{n}\qquad \text{for}\ n\ge 1,\ x\in\{-n,-n+2,...,n\}.
\end{align}
For $c,d\in\Z_{>0}$, let $I=[-c,d]\cap \Z$, and let $(K_n^{I,\mathrm{ERW}})_{n\ge0}$ denote the transition kernels of the ERW killed upon exiting $I$. Let $K^{I,\mathrm{SRW}}$ denote the transition kernel of the symmetric simple random walk (SRW) killed upon exiting the same interval $I$. Let $A_n$ be the set defined in (\ref{An-def}) for the ERW killed upon exiting $I$ and started from the origin. Then,
\begin{align}
  \label{ERW-eq3}
  \frac{{K}_n^{I,\mathrm{ERW}}(x,x\pm 1)}{K^{I,\mathrm{SRW}}(x,x\pm 1)}=1\pm \frac{(2p-1)x}{n}\qquad \text{for}\ x\in A_n,\ n\ge N.
\end{align}
To apply our results, we complete the rows corresponding to unreachable states as in (\ref{tildeK}). Taking $N=c\vee d$, $a(x,x\pm1)=\pm(2p-1)x$, and $b=0$, we obtain a sequence of transition kernels $(\widetilde{K}_n^{I,\mathrm{ERW}})_{n\ge0}$ satisfying (\ref{H1}) and (\ref{H2}).

Debs \cite{Debs-conditioning} showed that the eigenvalue $\lambda^{I,\mathrm{SRW}}$ and corresponding right eigenvector $R^{I,\mathrm{SRW}}$ are given by
\begin{align}
  \lambda^{I,\mathrm{SRW}}&=\cos \left(\frac{\pi}{c+d+2}\right),\\
  R^{I,\mathrm{SRW}}(x)&=\sin \left(\frac{(x+c+1)\pi}{c+d+2}\right)\qquad \text{for}\ x\in I.
\end{align}
By a straightforward calculation, the transition kernel $Q^{I,\mathrm{SRW}}$ of the h-transform with respect to $R^{I,\mathrm{SRW}}$ and its invariant distribution $\pi^{I,\mathrm{SRW}}$ are given by
\begin{align}
  Q^{I,\mathrm{SRW}}(x,x\pm 1)&=\frac{1}{2}\left\{1\mp \tan \left(\frac{\pi}{c+d+2}\right)\tan \left(\frac{\pi}{c+d+2}\left(x-\frac{d-c}{2}\right)\right)\right\},\\
  \pi^{I,\mathrm{SRW}}(x)&=\frac{2}{c+d+2}\sin^2\left(\frac{(x+c+1)\pi}{c+d+2}\right).
\end{align}
Moreover, condition (\ref{C3}) holds with
\begin{align}
  V(x)=\frac{2p-1}{2}x^2,
\qquad
\kappa=\frac{1-2p}{2}.
\end{align}
Thus, Theorems \ref{thm1.1} and \ref{thm1.2} apply to the completed kernels $(\widetilde{K}_n^{I,\mathrm{ERW}})$. By Remark \ref{Rem1.1}, the resulting conclusions also hold for the original ERW started from the origin. Since the process with transition kernels $(K_n^{I,\text{ERW}})$ is killed upon exiting the interval $I$, its lifetime coincides with the first exit time $\sigma_I$ of the original ERW from $I$. Therefore, in the case $m=0$ and $x=0$, (\ref{thm1.1-eq2}) takes the following form:
\begin{align*}
\label{ERW-eq2}
  \P_{(0,0)}^{\text{ERW}}(n<\sigma_I)&=\frac{4}{c+d+2}
\sum_{y\in E_{n\bmod 2}}
\sin \left(\frac{(y+c+1)\pi}{c+d+2}\right)\\
&\qquad \times\left(\cos\frac{\pi}{c+d+2}\right)^n
n^{\frac{1-2p}{2}}
h^{I,\text{ERW}}(0,0)
\Big(1+O(n^{-1})\Big).
\stepcounter{equation}\tag{\theequation}
\end{align*}
The Markov chain with transition kernels defined by (\ref{Qn-definition}) can therefore be interpreted as the ERW conditioned to stay in the interval $I$.

\begin{Rem}
For a fixed interval $I$, (\ref{ERW-eq2}) gives the precise asymptotic behavior of $\mathbb{P}_{(0,0)}^{\mathrm{ERW}}(n<\sigma_I)$ as $n\to\infty$. Related studies provide exit-time tail bounds that are uniform in the interval size or investigate asymptotics for expanding intervals; see \cite{AZ,CGP,Qin}.
\end{Rem}

\subsection{A Generalized P\'{o}lya Urn Conditioned to Have Bounded Differences in Color Counts}
Consider an urn containing balls of $r\ge2$ colors. At each time step, one ball is added to the urn, and its color is chosen according to probabilities determined by the current color counts. For $i=1,...,r$ and $n\ge 0$, let $U_{n,i}$ denote the number of balls of color $i$ in the urn at time $n$. Fix $\alpha\ge0$ and an initial state $u_0=(u_{0,1},\ldots,u_{0,r})\in\Z_{>0}^r$, and set $U_n:=(U_{n,1},\ldots,U_{n,r})$. Given the urn composition at time $n$, the next ball has color $i$ with conditional probability
\begin{align}
  \P_{u_0}\Big(U_{n+1}=U_n+{e}_i\mid \sigma(U_0,...,U_n)\Big)=\frac{U_{n,i}^\alpha}{U_{n,1}^\alpha+\cdots +U_{n,r}^\alpha},
\end{align}
where ${e}_i$ denotes the $i$-th unit vector. When $\alpha=0$, each color is chosen with probability $\frac{1}{r}$. When $\alpha=1$, the model reduces to the classical P\'{o}lya urn. For the classical P\'{o}lya urn, see, for example, Section 4.3.2 of \cite{Durrett}. We define
\begin{align}
  X_n:=\Big(U_{n,1}-U_{n,r},...,U_{n,r-1}-U_{n,r},0\Big)\in \Z^{r-1}\times \{0\}.
\end{align}
Since the total number of balls at time $n$ is $u_{0,1}+\cdots+u_{0,r}+n$, the vector $U_n$ is uniquely determined by $X_n$ and $n$. Consequently, $(X_n)_{n\ge0}$ is a time-inhomogeneous Markov chain on $\mathbb Z^{r-1}\times\{0\}$ with initial state $x_0:=(u_{0,1}-u_{0,r},\ldots,u_{0,r-1}-u_{0,r},0)$. Let
\begin{align}
  s_0&:=u_{0,1}+\cdots +u_{0,r},\\
  z_i(x)&:=x_i-\frac{x_1+\cdots +x_r}{r}\qquad \text{for}\ i=1,...,r,\ x\in \Z^r.
\end{align}
Since 
\begin{align}
\label{Pol-eq2}
  U_{n,i}=\frac{n+s_0}{r}+z_i(X_n),
\end{align}
for each state $x$ reachable at time $n$, the transition probabilities of $(X_n)_{n\ge0}$ are given by
\begin{align}
  K_n^{\text{P\'{o}l}}(x,x+v_i)=\frac{\displaystyle \left(\frac{n+s_0}{r}+z_i(x)\right)^\alpha}{\displaystyle \sum_{j=1}^r\left(\frac{n+s_0}{r}+z_j(x)\right)^\alpha}
\end{align}
for $i=1,...,r$, where $v_1:=e_1,...,\ v_{r-1}:=e_{r-1},\ v_r:=(-1,...,-1,0).$
For $M\ge 1$, we define
\begin{align}
  E_M:=\Big\{x\in \Z^{r-1}\times \{0\}:\ \max_{1\le i\le r}x_i-\min_{1\le i\le r}x_i\le M\Big\}.
\end{align}
Let $(K_n^{E_M,\text{P\'{o}l}})_{n\ge 0}$ denote the transition kernels of $(X_n)$ killed upon exiting $E_M.$ Let $K^{E_M,\mathrm{Unif}}$ denote the corresponding killed transition kernel for $\alpha=0$, given by
\begin{align}
K^{E_M,\text{Unif}}(x,x+v_i)=\frac{1}{r}\qquad \text{for}\ i=1,...,r,
\end{align}
whenever $x+v_i\in E_M.$ Let $A_n$ be the set of reachable states defined in (\ref{An-def}) for this killed chain started from $x_0$ at time $0$. By a straightforward calculation, we have
\begin{align*}
\label{Pol-eq1}
\frac{{K}_n^{E_M,\text{P\'{o}l}}(x,x+v_i)}{K^{E_M,\text{Unif}}(x,x+v_i)}&=\frac{\displaystyle \left(1+\frac{rz_i(x)}{n+s_0}\right)^\alpha}{\displaystyle \frac{1}{r}\sum_{j=1}^r\left(1+\frac{rz_j(x)}{n+s_0}\right)^\alpha}\\
&=1+\frac{\alpha rz_i(x)}{n+s_0}+\frac{\alpha(\alpha-1)r^2}{2(n+s_0)^2}\left(z_i(x)^2-\frac{1}{r}\sum_{j=1}^rz_j(x)^2\right)+O(n^{-3})
\stepcounter{equation}\tag{\theequation}
\end{align*}
for $x\in A_n$ and $i=1,...,r$ such that $x+v_i\in E_M$. By (\ref{Pol-eq1}), the coefficients in (\ref{H2}) are given by
\begin{align}
  a(x,x+v_i)&=\alpha rz_i(x),\\
  b(x,x+v_i)&=\frac{\alpha(\alpha-1)r^2}{2}\left(z_i(x)^2-\frac{1}{r}\sum_{j=1}^rz_j(x)^2\right)-s_0\alpha r z_i(x).
\end{align}
For sufficiently large $N$, the completed transition kernels $(\widetilde{K}_n^{E_M,\text{P\'{o}l}})$ defined by (\ref{tildeK}) satisfy assumptions (\ref{H1}) and (\ref{H2}). Moreover, a direct calculation shows that condition (\ref{C3}) holds with
\begin{align}
  V(x)=\frac{\alpha r}{2}\sum_{j=1}^r z_j(x)^2,\qquad \kappa=-\frac{\alpha(r-1)}{2}.
\end{align}
Thus, Theorems \ref{thm1.1} and \ref{thm1.2} apply to the completed kernels $(\widetilde{K}_n^{E_M,\text{P\'{o}l}})$. By Remark \ref{Rem1.1},
these conclusions also hold for the original killed chain started from $x_0$ at time $0$. The Markov chain with transition kernels defined by (\ref{Qn-definition}), started from $x_0$, can therefore be interpreted as the difference process $(X_n)$ conditioned to remain in $E_M$ for all time. Via (\ref{Pol-eq2}), this yields the generalized P\'{o}lya urn conditioned so that the difference between the largest and smallest color counts never exceeds $M$.

\subsection{Time-Inhomogeneous Wright--Fisher Model Conditioned Never to Be Absorbed}
In this subsection, we consider a Wright--Fisher model with a fixed population size $n_0\ge 3$ and a time-varying selective advantage. For Wright–Fisher models with time-dependent selection, see \cite{Donnelly-Weber,Weber}. Assume that each individual is of one of two types, $\mathsf{a}$ or $\mathsf{A}$. Let $\sigma_n$ denote the selective advantage of type $\mathsf{a}$ in generation $n$. Assume that $\sigma_n>-1$ for all $n\ge0$ and that \begin{align}
\label{def-sigman}
  \sigma_n=s_1+\frac{s_2}{n}+\frac{s_3}{n^2}+O(n^{-3})\qquad \text{for some}\  s_1>-1,\ s_2,s_3\in \R
\end{align}
holds. Let $X_n$ denote the number of individuals of type $\mathsf{a}$ in generation $n\ge0$, with $X_0=x_0\in\{1,...,n_0-1\}$. We assume that generation $n+1$ is formed by sampling with replacement from generation $n$, taking into account the selective advantage of type $\mathsf{a}$. Conditional on $X_n=i$, the random variable $X_{n+1}$ has a binomial distribution with parameters $n_0$ and $p_n(i)$. Thus, the one-step transition probabilities are given by
\begin{align}
  K_n^{\text{WF}}(i,j)=\binom{n_0}{j}p_n(i)^j\Big(1-p_n(i)\Big)^{n_0-j}\qquad \text{for}\ n\ge 0,\ i,j\in \{0,...,n_0\},
\end{align}
where 
\begin{align}
  p_n(i):=\frac{(1+\sigma_n)i}{n_0+\sigma_ni}\qquad \text{for}\ i=0,...,n_0.
\end{align}
When $\sigma_n\equiv 0$, this model reduces to the neutral Wright--Fisher model. Let $(K_n^{\text{Kill,WF}})_{n\ge 0}$ denote the transition kernels of the process $(X_n)_{n\ge 0}$ killed upon reaching either of the absorbing states $0$ and $n_0$. Let $K^{\text{Kill,WF}}$ denote the corresponding killed transition kernel for $\sigma_n\equiv s_1>-1$, given by
\begin{align}
  K^{\text{Kill,WF}}(i,j)=\binom{n_0}{j}\frac{(1+s_1)^ji^j(n_0-i)^{n_0-j}}
{(n_0+s_1i)^{n_0}},
\qquad \text{for}\ i,j\in \{1,...,n_0-1\}.
\end{align}
By Taylor expansion, we have
\begin{align*}
\label{WF-eq2}
&\log\left(\frac{K_n^{\text{Kill,WF}}(i,j)}{ K^{\text{Kill,WF}}(i,j)}\right)\\
&\qquad =\log\left\{\left(\frac{1+\sigma_n}{1+s_1}\right)^j
\left(\frac{n_0+s_1i}{n_0+\sigma_ni}\right)^{n_0}\right\}\\
&\qquad =\frac{s_2}{n}\left(\frac{j}{1+s_1}-\frac{n_0i}{n_0+s_1i}\right)\\
&\qquad \qquad+\frac{1}{n^2}\left\{s_3\left(\frac{j}{1+s_1}-\frac{n_0i}{n_0+s_1i}\right)+\frac{s_2^2}{2}\left(-\frac{j}{(1+s_1)^2}+\frac{n_0i^2}{(n_0+s_1i)^2}\right)\right\}+O(n^{-3}).
\stepcounter{equation}\tag{\theequation}  
\end{align*}
Since all entries of $K_n^{\mathrm{Kill,WF}}$ and $K^{\mathrm{Kill,WF}}$ are strictly positive, (\ref{H1}) holds. Exponentiating (\ref{WF-eq2}), we find that assumption (\ref{H2}) holds with
\begin{align}
  a(i,j)&=s_2\left(\frac{j}{1+s_1}-\frac{n_0i}{n_0+s_1i}\right),\\
  b(i,j)&=s_3\left(\frac{j}{1+s_1}-\frac{n_0i}{n_0+s_1i}\right)+\frac{s_2^2}{2}\left(\left(\frac{j}{1+s_1}-\frac{n_0i}{n_0+s_1i}\right)^2-\frac{j}{(1+s_1)^2}+\frac{n_0i^2}{(n_0+s_1i)^2}\right).
\end{align}
Moreover, we have
\begin{align}
\label{WF-eq1}
  \text{Condition (\ref{C3}) holds}\qquad \Longleftrightarrow\qquad s_1s_2=0.
\end{align}
Indeed, if $s_1s_2=0$, then direct calculation shows that condition $(\ref{C3})$ holds with
\begin{align}
  V(i)=s_2i,\qquad \kappa=0.
\end{align}
Conversely, condition (\ref{C3}) implies
\begin{align}
  \kappa=a(i,i)=-\frac{s_1s_2i(n_0-i)}{(1+s_1)(n_0+s_1i)}.
\end{align}
For the right-hand side to be independent of $i$, it is necessary that $s_1s_2=0$. Thus, (\ref{WF-eq1}) is proved.

Therefore, Theorems \ref{thm1.1} and \ref{thm1.2} apply to the kernels $(K_n^{\text{Kill,WF}})$. Since the states $0$ and $n_0$ are both absorbing states of $X$, the Markov chain with transition kernels defined by (\ref{Qn-definition}), started from $x_0$, can be interpreted as the time-inhomogeneous Wright--Fisher model conditioned never to be absorbed at either boundary.

\begin{Rem}
The Wright--Fisher model conditioned to avoid absorption in the time-homogeneous setting has been studied by Huillet \cite{Huillet}. That paper also considers several other types of conditioning and quasi-stationary distributions.
\end{Rem}

\subsection{Time-Inhomogeneous Moran Model Conditioned Never to Be Absorbed}
In the Wright--Fisher model considered in the preceding subsection, the entire population was replaced in each generation. In contrast, we consider the Moran model, in which one individual is selected in each generation and replaced according to the current proportions of types $\mathsf{a}$ and $\mathsf{A}$ in the population. For Moran models with time-dependent selection, see \cite{Chalub-Souza}.

Let $n_0\ge 3$ be a population size, and let $\sigma_n>-1$ denote the selective advantage of type $\mathsf{a}$ in generation $n$. Assume that (\ref{def-sigman}) is satisfied. Let $X_n$ denote the number of individuals of type $\mathsf{a}$ in generation $n\ge 0$, with $X_0=x_0\in \{1,...,n_0-1\}$. Then the transition kernels of $(X_n)_{n\ge 0}$ are given by
\begin{align}
K_n^{\text{Mor}}(i,i+1)&=\frac{n_0-i}{n_0}\cdot \frac{(1+\sigma_n)i}{n_0+\sigma_n i},\\
K_n^{\text{Mor}}(i,i-1)&=\frac{i}{n_0}\cdot \frac{n_0-i}{n_0+\sigma_n i},\\
K_n^{\text{Mor}}(i,i)&=\frac{(n_0-i)^2+(1+\sigma_n)i^2}{n_0(n_0+\sigma_n i)}.
\end{align}
When $\sigma_n\equiv 0$, this model reduces to the neutral Moran model. Let $(K_n^{\text{Kill,Mor}})_{n\ge 0}$ denote the transition kernels of the process $(X_n)$ killed upon reaching either of the absorbing states $0$ and $n_0$. Let $K^{\text{Kill,Mor}}$ denote the corresponding killed transition kernel for $\sigma_n\equiv s_1>-1$. Then, we have
\begin{align}
\label{Moran-eq2}
  \frac{K_n^{\text{Kill,Mor}}(i,j)}{K^{\text{Kill,Mor}}(i,j)}=1+\frac{s_2A(i,j)}{n}
+\frac{\left(s_3-\frac{s_2^2i}{n_0+s_1i}\right)A(i,j)}{n^2}
+O(n^{-3}),\qquad \text{for}\ j=i,i\pm 1,
\end{align}
where
\begin{align}
A(i,j):=
\begin{cases}
\displaystyle \frac{1}{1+s_1}-\frac{i}{n_0+s_1i},
&\text{for}\ j=i+1,\\[10pt]
\displaystyle -\frac{i}{n_0+s_1i},
&\text{for}\ j=i-1,\\[10pt]
\displaystyle \frac{i^2}{(n_0-i)^2+(1+s_1)i^2}-\frac{i}{n_0+s_1i},
&\text{for}\ j=i.
\end{cases}
\end{align}
Since $\sigma_n>-1$, the killed kernels are irreducible and have the same support as $K^{\mathrm{Kill,Mor}}$, so (\ref{H1}) holds. Moreover, (\ref{Moran-eq2}) yields (\ref{H2}) with
\begin{align}
a(i,j)=s_2A(i,j),\qquad
b(i,j)=\left(s_3-\frac{s_2^2i}{n_0+s_1i}\right)A(i,j).
\end{align}
Moreover, we have
\begin{align}
\label{Moran-eq1}
  \text{Condition (\ref{C3}) holds}\qquad \Longleftrightarrow\qquad s_2=0.
\end{align}
Indeed, if $s_2=0$, then condition (\ref{C3}) holds with $V\equiv0$ and $\kappa=0$. Conversely, condition (\ref{C3}) implies $\kappa=a(i,i)=s_2A(i,i)$.  For the right-hand side to be independent of $i$, it is necessary that $s_2=0$. Thus, (\ref{Moran-eq1}) is proved.

Therefore, Theorems \ref{thm1.1} and \ref{thm1.2} apply to the kernels $(K_n^{\text{Kill,Mor}})$. Since the states $0$ and
$n_0$ are both absorbing states of $X$, the Markov chain with transition kernels defined by
(\ref{Qn-definition}), started from $x_0$, can be interpreted as the time-inhomogeneous Moran model conditioned never to be absorbed at either boundary.

\begin{Rem}
The time-homogeneous neutral Moran model conditioned never to be absorbed has been studied by Huillet \cite{Huillet-Moran}. Moreover, Chalub et al. \cite{Moran} use the conditioned process to study several equations.
\end{Rem}

\subsection{Noncolliding Time-Inhomogeneous Random Walks Conditioned to Stay in an Interval}
Consider $r\ge 1$ particles on the integer lattice $\Z$. At each time step, one particle is selected at random and moves one unit to the left or right. The process is killed whenever two particles collide. This model was introduced by Fisher \cite{Fisher}. In this subsection, we consider a model in which the particle movement probabilities depend on time.

For $i=1,...,r$, let $\xi_{n,i}$ denote the position of particle $i$ at time $n$. Assume that the initial configuration is given by $\xi_{0,i}=x_i$, where $x_1<\cdots <x_r$. At each time step $n$, one particle is selected uniformly at random and moves one unit to the right with probability $p_n$ or one unit to the left with probability $1-p_n$. The process is killed upon a collision between any two particles. Assume that $0<p_n<1$ for all $n\ge 0$ and that
\begin{align}
  p_n=p+\frac{\alpha}{n}+\frac{\beta}{n^2}+O(n^{-3}),\qquad \text{for some}\ 0<p<1,\ \alpha,\beta\in \R.
\end{align}
Set $X_n:=(\xi_{n,1},...,\xi_{n,r})$ and define
\begin{align}
  E_M:=\Big\{x\in \Z^r:\ -M\le x_1<\cdots <x_r\le M\Big\}\qquad \text{for}\ M\ge 1.
\end{align}
Then the transition kernels of $(X_n)_{n\ge 0}$ killed upon exiting $E_M$ are given by
\begin{align}
  K_n^{\text{Kill,NCTI}}(x,x+e_i)=\frac{p_n}{r},\qquad K_n^{\text{Kill,NCTI}}(x,x-e_i)=\frac{1-p_n}{r},
\end{align}
whenever $x,x+e_i\in E_M$ and $x,x-e_i\in E_M$, respectively. Let $K^{\text{Kill,NC}}$ denote the corresponding killed transition kernel for $p_n\equiv p$. This is the classical model introduced in \cite{Fisher}, and the case $r=1$ corresponds to the gambler's ruin model. Since $0<p_n<1$ and $0<p<1$, the kernels $(K_n^{\mathrm{Kill,NCTI}})$ and $K^{\mathrm{Kill,NC}}$ have the same support. Thus, assumption (\ref{H1}) holds. For every permitted transition, we have
\begin{align}
  \frac{K_n^{\text{Kill,NCTI}}(x,x+e_i)}{K^{\text{Kill,NC}}(x,x+e_i)}&=1+\frac{\alpha}{pn}+\frac{\beta}{pn^2}+O(n^{-3}),\\
  \frac{K_n^{\text{Kill,NCTI}}(x,x-e_i)}{K^{\text{Kill,NC}}(x,x-e_i)}&=1-\frac{\alpha}{(1-p)n}-\frac{\beta}{(1-p)n^2}+O(n^{-3}).
\end{align}
Thus, assumption (\ref{H2}) holds with
\begin{align}
a(x,x+e_i)=\frac{\alpha}{p}&,
\qquad a(x,x-e_i)=-\frac{\alpha}{1-p},\\
b(x,x+e_i)=\frac{\beta}{p}&,
\qquad b(x,x-e_i)=-\frac{\beta}{1-p}.
\end{align}
Moreover, a direct calculation shows that condition (\ref{C3}) holds with
\begin{align}
  V(x)=\frac{\alpha}{2p(1-p)}\sum_{i=1}^r x_i,\qquad \kappa=\frac{\alpha(1-2p)}{2p(1-p)}.
\end{align}
Thus, Theorems \ref{thm1.1} and \ref{thm1.2} apply to the kernels $(K_n^{\text{Kill,NCTI}})$. The
Markov chain with transition kernels defined by (\ref{Qn-definition}), started from $(x_1,...,x_r)$, can be interpreted as the system of time-inhomogeneous random walks conditioned never to collide and to remain in the interval $\{-M,\ldots,M\}$ for all time.

\begin{Rem}
In the time-homogeneous setting, Grabiner \cite{Grabiner} obtained determinant formulas for the number of noncolliding paths confined to a finite interval, while Krattenthaler \cite{Krattenthaler} derived the asymptotic behavior of the number of such paths as the number of steps tends to infinity.
\end{Rem}


\section*{Acknowledgement}
The author is deeply grateful to Kouji Yano, Syota Esaki, and Go Tokumitsu for their insightful comments, from which this work benefited greatly. The author acknowledges support from JSPS KAKENHI Grant Number 26KJ1605.

\section*{AI Use Statement}
The author used AI tools (ChatGPT 6.1) in the preparation of this manuscript. Specifically, AI was used to assist with mathematical discussions, improve and proofread the English writing, and conduct literature searches. All mathematical results and the final manuscript were written by the author, who takes full responsibility for the content of this paper.

\addcontentsline{toc}{section}{References}
\bibliographystyle{plain}

\begin{thebibliography}{10}

\bibitem{AZ}
M.~Andr^^c3^^a9 and L.~Zuazn^^c3^^a1bar.
\newblock Estimates on escape times for the elephant random walk, preprint, arXiv: 2602.18953.

\bibitem{Moran}
F.~A. C.~C. Chalub, L.~Monsaingeon, A.~Ribeiro, and M.~O. Souza.
\newblock Gradient flow formulations of discrete and continuous evolutionary models: a unifying perspective.
\newblock {\em Acta Appl. Math.}, 171:Paper No. 24, 50, 2021.

\bibitem{Chalub-Souza}
F.~A. C.~C. Chalub and M.~O. Souza.
\newblock On the stochastic evolution of finite populations.
\newblock {\em J. Math. Biol.}, 75(6-7):1735--1774, 2017.

\bibitem{Champagnat-Villemonais-2}
N.~Champagnat and D.~Villemonais.
\newblock Exponential convergence to quasi-stationary distribution and {$Q$}-process.
\newblock {\em Probab. Theory Related Fields}, 164(1-2):243--283, 2016.

\bibitem{Champagnat-Villeminais}
N.~Champagnat and D.~Villemonais.
\newblock Uniform convergence of penalized time-inhomogeneous {M}arkov processes.
\newblock {\em ESAIM Probab. Stat.}, 22:129--162, 2018.

\bibitem{CGP}
C.~F. Coletti, R.~de~M.~Grisi, and I.~Papageorgiou.
\newblock A spectral approach to survival bounds for the elephant random walk, preprint, arXiv: 2606.08884.

\bibitem{Debs-conditioning}
P.~Debs.
\newblock Penalisation of the symmetric random walk by several functions of the supremum.
\newblock {\em Markov Process. Related Fields}, 18(4):651--680, 2012.

\bibitem{Donnelly-Weber}
P.~Donnelly and N.~Weber.
\newblock The {W}right-{F}isher model with temporally varying selection and population size.
\newblock {\em J. Math. Biol.}, 22(1):21--29, 1985.

\bibitem{Durrett}
R.~Durrett.
\newblock {\em Probability: theory and examples}, volume~31 of {\em Cambridge Series in Statistical and Probabilistic Mathematics}.
\newblock Cambridge University Press, Cambridge, fourth edition, 2010.

\bibitem{Fisher}
M.~E. Fisher.
\newblock Walks, walls, wetting, and melting.
\newblock {\em J. Statist. Phys.}, 34(5-6):667--729, 1984.

\bibitem{Grabiner}
D.~J. Grabiner.
\newblock Random walk in an alcove of an affine {W}eyl group, and non-colliding random walks on an interval.
\newblock {\em J. Combin. Theory Ser. A}, 97(2):285--306, 2002.

\bibitem{Huillet}
T.~E. Huillet.
\newblock On discrete-time multiallelic evolutionary dynamics driven by selection.
\newblock {\em J. Probab. Stat.}, pages Art. ID 580762, 27, 2010.

\bibitem{Huillet-Moran}
T.~E. Huillet.
\newblock Siegmund duality with applications to the neutral {M}oran model conditioned on never being absorbed.
\newblock {\em J. Phys. A}, 43(37):375001, 14, 2010.

\bibitem{Klenke}
A.~Klenke.
\newblock {\em Probability theory---a comprehensive course}.
\newblock Universitext. Springer, Cham, third edition, 2020.

\bibitem{Knight}
F.~B. Knight.
\newblock Brownian local times and taboo processes.
\newblock {\em Trans. Amer. Math. Soc.}, 143:173--185, 1969.

\bibitem{Krattenthaler}
C.~Krattenthaler.
\newblock Asymptotics for random walks in alcoves of affine {W}eyl groups.
\newblock {\em S\'em. Lothar. Combin.}, 52:Art. B52i, 72, 2004/07.

\bibitem{Lambert}
A.~Lambert.
\newblock Completely asymmetric {L}\'evy processes confined in a finite interval.
\newblock {\em Ann. Inst. H. Poincar\'e{} Probab. Statist.}, 36(2):251--274, 2000.

\bibitem{Qin}
S.~Qin.
\newblock Cover times and ranges of elephant random walks, preprint, arXiv: 2609.17264.

\bibitem{Schutz-Trimper}
G.~M. Sch\"utz and S.~Trimper.
\newblock Elephants can always remember: Exact long-range memory effects in a non-{M}arkovian random walk.
\newblock {\em Physical Review E}, 70(4):045101, 10 2004.

\bibitem{Weber}
N.~C. Weber.
\newblock The {W}right-{F}isher model with varying selection.
\newblock {\em J. Appl. Probab.}, 23(2):504--508, 1986.

\end{thebibliography}

\end{document}